\documentclass[12pt]{article}
\usepackage{latexsym,amsfonts,amsmath,graphics,amsthm}
\usepackage{epsfig}
\usepackage{mathrsfs}

\newtheorem{theorem}{Theorem}
\newtheorem{lemma}{Lemma}

\newtheorem{proposition}{Proposition}

\newtheorem{definition}{Definition}

\newtheorem{remark}{Remark}

\newcommand{\bsalpha}{\boldsymbol{\alpha}}

\newcommand{\bsk}{\boldsymbol{k}}

\newcommand{\bsx}{\boldsymbol{x}}

\newcommand{\bsg}{\boldsymbol{g}}

\newcommand{\bsp}{\boldsymbol{p}}

\newcommand{\bsy}{\boldsymbol{y}}

\newcommand{\bszero}{\boldsymbol{0}}

\newcommand{\rd}{\,\mathrm{d}}

\newcommand{\NN}{\mathbb{N}}
\newcommand{\ZZ}{\mathbb{Z}}

\newcommand{\PP}{\mathbb{P}}

\newcommand{\cH}{{\mathscr H}}

\newcommand{\uu}{\mathfrak{u}}

\makeatletter
\newcommand{\rdots}{\mathinner{\mkern1mu\lower-1\p@\vbox{\kern7\p@\hbox{.}}
\mkern2mu \raise4\p@\hbox{.}\mkern2mu\raise7\p@\hbox{.}\mkern1mu}}
\makeatother

\allowdisplaybreaks

\begin{document}

\title{The $\bsp$-adic diaphony of the Halton sequence}

\author{Friedrich Pillichshammer\thanks{The author is partially supported by the Austrian Science Foundation (FWF), Project S9609, that is part of the Austrian National Research Network "Analytic Combinatorics and Probabilistic Number Theory". }}

\date{}
\maketitle

\begin{abstract}
The $\bsp$-adic diaphony as introduced by Hellekalek is a quantitative measure for the irregularity of distribution of a sequence in the unit cube. In this paper we show how this notion of diaphony can be interpreted as worst-case integration error in a certain reproducing kernel Hilbert space. Our main result is an upper bound on the $\bsp$-adic diaphony of the Halton sequence.
\end{abstract}

\vspace{.5cm}



\section{Introduction}

In many applications, like numerical integration using Monte Carlo or quasi-Monte Carlo algorithms, where random number generators or 
low discrepancy sequences are used, the success of the algorithm often depends on the distribution properties of the underlying point set.  Consequently various measures for the irregularity of the distribution of sequences in the unit cube have been introduced and analyzed. Some of them stem from numerical integration where the worst-case integration error has been analyzed, others are based in geometrical concepts or on specific function systems, see, for example, \cite{DP10,dt,kuinie,matou,niesiam}.

For a function space $\cH$ of functions $f$ defined on $[0,1]^s$ with norm $\|\cdot\|$ the worst-case error
$e(\cH,\omega)$ using a quasi-Monte Carlo rule $\frac{1}{N} \sum_{n=0}^{N-1} f(\bsx_n)$ based on a sequence $\omega = (\bsx_n)_{n \in \NN_0}$ in the unit cube $[0,1)^s$,
is given by 
\begin{equation}\label{wce}
e(\cH,\omega) := \sup_{f\in \cH \atop \|f\| \le 1} \left|\int_{[0,1]^s} f(\bsx)\rd\bsx - \frac{1}{N}\sum_{n=0}^{N-1} f(\bsx_n) \right|.
\end{equation}
For a given function 
space and norm, this worst-case error then only depends on the sequence used. In some cases this worst-case error can be related to the discrepancy of the sequence which is a geometric measure for the irregularity of the distribution of a sequence, see \cite{DP10,Hick,SW98}.

In this paper we deal with a further measure of the irregularity of distribution which is called diaphony and which is based on certain function systems. The classical diaphony introduced by Zinterhof~\cite{zint} (see also \cite[Definition~1.29]{dt} or \cite[Exercise~5.27, p.~162]{kuinie}) is based on the trigonometric function system. Later the concept of dyadic diaphony, which is based on Walsh functions in base $b=2$, was introduced by Hellekalek and Leeb \cite{helleeb}. This concept has been generalized by Grozdanov and Stoilova \cite{grozstoi} to general integer bases $b \ge 2$. Although these diaphonies are quantitative measures for the irregularity of distribution of arbitrary sequences each of them is particularly suited to analyze a special class of sequences. For example, the classical diaphony is suitable to analyze $(n\bsalpha)$-sequences and lattice point sets and the diaphony based on Walsh functions is especially useful to analyze $(t,s)$-sequences and $(t,m,s)$-nets in suitable bases. 

Quite recently Hellekalek~\cite{hel2010} introduced a further notion of diaphony which is based on the $p$-adic function system. This notion of so-called $\bsp$-adic diaphony is especially useful to analyze distribution properties of the Halton sequence.

The exact definition of the the $p$-adic function system and of $\bsp$-adic diaphony according to \cite{hel2010} will be presented in the next section. In Section~\ref{int} we show how the $\bsp$-adic diaphony can be interpreted as the worst-case integration error of functions from a certain reproducing kernel Hilbert space. The main result of this paper is presented in Section~\ref{dia_halt} where we estimate the $\bsp$-adic diaphony of the Halton sequence.

\section{Definition of $\bsp$-adic diaphony}

In this section we present the definition of $\bsp$-adic diaphony as introduced by Hellekalek~\cite{hel2010}. Before we do so we need to introduce some notation. We follow \cite[Section~2]{hel2010} and \cite[Section~2]{hel2009}.

Let $\PP$ denote the set of prime numbers. For $p \in \PP$ we define the set of $p$-adic numbers as the set of formal sums
\begin{equation*}
\mathbb{Z}_p = \left\{z = \sum_{r=0}^\infty z_r p^r\, : \, z_r \in \{0,\ldots,p-1\} \mbox{ for all } r \in \NN_0\right\}.
\end{equation*}
The set $\mathbb{N}_0$ of non-negative integers is a subset of $\mathbb{Z}_p$. For two non-negative integers $y, z \in \mathbb{Z}_p$, the sum $y+z \in \mathbb{Z}_p$ is defined as the usual sum of integers. The addition can be extended to all $p$-adic numbers and, with this addition, $\ZZ_p$ forms an abelian group.

Define the so-called {\it Monna-map} $\phi_p:\mathbb{Z}_p \to [0,1)$ by
\begin{equation*}
\phi_p(z) = \sum_{r=0}^\infty z_r p^{-r-1} \pmod{1}.
\end{equation*}
We also define the inverse $\phi_p^+: [0,1)\to \mathbb{Z}_p$ by
\begin{equation*}
\phi_p^+\left(\sum_{r=0}^\infty x_r p^{-r-1}\right) = \sum_{r=0}^\infty x_r p^r,
\end{equation*}
where we always use the finite $p$-adic representation for $p$-adic rationals in $[0,1)$.

For $k \in \mathbb{N}_0$ we can define characters $\chi_k:\mathbb{Z}_p \to \{c \in \mathbb{C}\,:\, |c| = 1\}$ of $\mathbb{Z}_p$ by
\begin{equation*}
\chi_k(z) = \exp(2\pi \mathrm{i} \phi_p(k) z).
\end{equation*}
These functions satisfy $\chi_k(y+z) = \chi_k(y) \chi_k(z)$, $\chi_k(0) = 1$, $\chi_0(z) = 1$, $\chi_k(z) \chi_l(z) = \chi_{\phi_p^+(\phi_p(k) + \phi_p(l) \pmod{1})}(z)$.

Let $\gamma_k: [0,1) \to \{c \in \mathbb{C}\,:\, |c| = 1\}$ where
\begin{equation*}
\gamma_k(x) = \chi_k(\phi_p^+(x)).
\end{equation*}
We have $\gamma_k(x) \gamma_l(x) = \gamma_{\phi_p^+(\phi_p(k) + \phi_p(l) \pmod{1})}(x)$ and $\overline{\gamma_k(x)}=\gamma_{\phi_p^+(-\phi_p(k) \pmod{1})}(x)$. We call $\gamma_k$ the {\it $k$-th $p$-adic function}.

For $\bsp=(p_1,\ldots,p_s)\in \PP^s$, $\bsk=(k_1,\ldots,k_s) \in \NN_0^s$ and for $\bsx=(x_1,\ldots,x_s) \in [0,1)^s$ define the {\it $k$-th $\bsp$-adic function} by $$\gamma_{\bsk}(\bsx):=\prod_{i=1}^s \gamma_{k_i}(x_i).$$

\begin{remark}[ONB property]\label{ONB_prop}\rm
It has been shown by Hellekalek \cite[Corollary~3.10]{hel2010} that the system $\{\gamma_{\bsk}\, : \, \bsk \in \NN_0^s\}$ is an orthonormal basis of $L_2([0,1]^s)$.
\end{remark}

For $\bsp=(p_1,\ldots,p_s) \in \PP^s$ and for $\bsk=(k_1,\ldots,k_s)\in \NN_0^s$ we put $\rho_{\bsp}(\bsk)=\prod_{j=1}^s \rho_{p_j}(k_j)$ where for $p \in \PP$ we put $\rho_p(0)=1$ and $\rho_p(k)=p^{-2 t}$ for $k \in \NN$ satisfying $p^t \le k < p^{t+1}$ for some $t \in \NN_0$. 

Now we can state the formal definition of $\bsp$-adic diaphony according to Hellekalek~\cite{hel2010}.

\begin{definition}[Hellekalek~\cite{hel2010}]\rm
Let $s \in \NN$ and $\bsp \in \PP^s$. The {\it $\bsp$-adic diaphony} of a sequence $\omega =(\bsx_n)_{n \in \NN_0}$ in $[0,1)^s$ is defined as
$$F_N(\omega)=\left(\frac{1}{\sigma_{\bsp}-1} \sum_{\bsk \in \NN_0^s\setminus\{\bszero\}} \rho_{\bsp}(\bsk) \left|\frac{1}{N}\sum_{n=0}^{N-1}\gamma_{\bsk}(\bsx_n)\right|^2\right)^{1/2},$$ where $\sigma_{\bsp}:=\prod_{i=1}^s (p_i+1)$. 
\end{definition}

Note that the $\bsp$-adic diaphony is normalized, i.e., for any sequence $\omega$ and for any $N \in \NN$ we have $0 \le F_N(\omega)\le 1$. It has been shown in \cite[Theorem~ 3.14]{hel2010} that the $\bsp$-adic diaphony is a quantitative measure for the irregularity of distribution modulo one of a sequence. In fact, a sequence $\omega$ is uniformly distributed modulo one if and only if $\lim_{N \rightarrow \infty}F_N(\omega)=0$. In \cite[Theorem~3.16]{hel2010} it has been shown that for $\bsp=(p,\ldots,p)$ the $\bsp$-adic diaphony of a regular lattice consisting of $N=p^{gs}$ elements is of order $(\log N)^{1/2}/N^{1/s}$. In Section~\ref{dia_halt} we will show that the $\bsp$-dic diaphony of the first $N$ elements of the Halton sequence is of order $(\log N)^{s/2}/N$.

\section{The $\bsp$-adic diaphony and quasi-Monte Carlo integration}\label{int}

Define the function $$K_{\bsp,s}(\bsx,\bsy):=\sum_{\bsk \in \NN_0^s} \rho_{\bsp}(\bsk) \gamma_{\bsk}(\bsx) \overline{\gamma_{\bsk}(\bsy)}.$$ Then we can write the $\bsp$-adic diaphony as 
\begin{equation}\label{err_dia1}
F_N(\omega)=\left(\frac{1}{\sigma_{\bsp}-1} \left(-1+\frac{1}{N^2} \sum_{n,m=0}^{N-1} K_{\bsp,s}(\bsx_n,\bsx_m)\right)\right)^{1/2}.
\end{equation}

We define an inner product by 
\[
  \langle f, g \rangle_{\bsp, s} 
= \sum_{\bsk \in \NN_0^s} \rho_{\bsp}(\bsk)^{-1}  \widehat{f}(\bsk)  \overline{\widehat{g}(\bsk)},
\]
with $\bsk = (k_1,\ldots,k_s)$ and
\[
   \widehat{f}(\bsk) := \int_{[0,1]^s} f(\bsx)  \overline{\gamma_{\bsk}(\bsx)} \rd \bsx.
\]
A norm is given by $||f||_{\bsp,s} := \langle f, f \rangle_{\bsp,s}^{1/2}$. (In the sequel we omit the index $s$ whenever $s=1$.) Now $K_{\bsp,s}$ can be shown to be the reproducing kernel of the function space
\[
    \cH_{\bsp, s} = \cH_{p_1} \otimes \ldots \otimes \cH_{p_s}
\]
which is the $s$-fold tensor product of function spaces of the form
\[ 
 \cH_{p} := \{ f \, : \, ||f||_{p} < \infty\}.
\]
From the general theory of reproducing kernel Hilbert spaces (see, for example, \cite{A50}) it is known that it suffices to prove this for the one-dimensional case. Let 
\[
    K_{p}(x,y) 
:=  \sum_{k=0}^\infty \rho_p(k) \gamma_k(x) \overline{\gamma_k(y)}
\]
and note that we have $K_{p}(x,y) = \overline{K_{p}(y,x)}$. In fact, the kernel $K_p$ is a real function as $\overline{\gamma_k(x)}=\gamma_{\phi_p^+(-\phi_p(k)\pmod{1})}(x)$ and $\rho_p(k) = \rho_p(\phi_p^+(-\phi_p(k)\pmod{1}))$.
Hence $K_{p}(x,y) = K_{p}(y,x)$. 

We have $K_{p}(\cdot,y) \in \cH_{p}$ as
\[
   ||K_{p}(\cdot,y)||_{p}^2 
=  \sum_{k=0}^\infty \rho_p(k)
=  1 + p < \infty.
\]
Further we have 
\[
  \langle f, K_{p}(\cdot,y) \rangle_{p} 
=  \sum_{k=0}^\infty \widehat{f}(k) \gamma_k(y) = f(y).
\]
Therefore $K_{p}$ is the reproducing kernel of the space $\cH_{p}$. Since $$K_{\bsp,s}(\bsx,\bsy)=\prod_{i=1}^s K_{p_i}(x_i,y_i)$$ it follows that $K_{\bsp,s}$ is the reproducing kernel of $\cH_{\bsp,s}$.

Using \cite[Proposition~2.11]{DP10} it follows that the squared worst-case integration error of functions from $\cH_{\bsp, s}$ is given by 
\begin{equation}\label{err_dia2}
e^2(\cH_{\bsp,s},\omega)=-1+\frac{1}{N^2} \sum_{n,m=0}^{N-1} K_{\bsp,s}(\bsx_n,\bsx_m).
\end{equation}

Combining \eqref{err_dia1} and \eqref{err_dia2} we obtain the following result.

\begin{proposition}
Let $s \in \NN$ and $\bsp \in \PP^s$. Then the worst-case integration error in $\cH_{\bsp,s}$ and the $\bsp$-adic diaphony of a sequence $\omega$ in $[0,1)^s$ are related by $$e(\cH_{\bsp,s},\omega)=\sqrt{\sigma_{\bsp}-1} F_N(\omega).$$
\end{proposition}

We now show that the reproducing kernel $K_p$ can be written in a concise form. Let $\mathrm{e}(x) := \exp(2\pi \mathrm{i} x)$. Note that for $x \in [0,1)$ and $n \in \mathbb{Z}$ we have $\mathrm{e}(x+n) = \mathrm{e}(x)$.

We have $\gamma_0(x) =1$. For $k = \kappa_0 + \cdots + \kappa_a p^a \in \mathbb{N}$ and $x = x_1 p^{-1} + x_2 p^{-2} + \cdots \in [0,1)$ we have
\begin{eqnarray*}
\gamma_k(x) & = & \mathrm{e}\left(\left(\kappa_0 p^{-1} + \cdots + \kappa_a p^{-a-1} \right) \left(x_1 + x_2 p + \cdots\right)\right) \\ & = & \mathrm{e}\left(\frac{1}{p} (\kappa_0x_1 + \cdots + \kappa_a x_{a+1}) + \cdots + \frac{1}{p^{a+1}} \kappa_a x_1 \right).
\end{eqnarray*}

Let $a \in \mathbb{N}$ and $0 \le l < p$ be fixed, then
\begin{eqnarray*}
\lefteqn{ \sum_{k= l p^{a}}^{(l+1) p^{a}-1} \gamma_k(x) } \\ & = & \mathrm{e}(l \left(x_{a+1} p^{-1} + \cdots +x_1 p^{-a-1}\right)) \\ && \times \sum_{\kappa_0=0}^{p-1} \mathrm{e}(\kappa_0 x_1 p^{-1}) \sum_{\kappa_1=0}^{p-1} \mathrm{e}(\kappa_1 \left(x_2 p^{-1} + x_1 p^{-2} \right)) \cdots \sum_{\kappa_{a-1}=0}^{p-1} \mathrm{e}(\kappa_{a-1} \left(x_a p^{-1} + \cdots + x_1 p^{-a}\right)) \\ & = & \left\{ \begin{array}{ll} \mathrm{e}(l x_{a+1} p^{-1}) p^{a} & \mbox{if } x_1 = \cdots = x_{a} = 0, \\ 0 & \mbox{otherwise}.  \end{array} \right.
\end{eqnarray*}
Hence
\begin{eqnarray*}
\sum_{k= l p^{a}}^{(l+1) p^{a}-1} \gamma_k(x) \overline{\gamma_k(y)} = \left\{ \begin{array}{ll} \mathrm{e}(l (x_{a+1}-y_{a+1}) p^{-1}) p^{a} & \mbox{if } x_1 = y_1, \ldots , x_{a} = y_a, \\ 0 & \mbox{otherwise}.  \end{array} \right.
\end{eqnarray*}

Now we obtain
\begin{eqnarray*}
K_{p}(x,y) & = & 1+ \sum_{a=0}^{\infty} \frac{1}{p^{2 a }} \sum_{l=1}^{p-1} \sum_{k=l p^{a}}^{(l+1)p^{a}-1}  \gamma_k(x) \overline{\gamma_k(y)}.
\end{eqnarray*}
If $x=y$, then we have 
\begin{eqnarray*}
K_{p}(x,y) & = & 1+ \sum_{a=0}^{\infty} \frac{1}{p^{2a }} \sum_{l=1}^{p-1} p^{a} = 1+p.
\end{eqnarray*}
If $x \not=y$, more precisely, if $x_i=y_i$ for $i=1,\ldots,i_0 -1$ and $x_{i_0}\not=y_{i_0}$, then we have 
\begin{eqnarray*}
K_{p}(x,y) & = & 1+ \left(\sum_{a=0}^{i_0-2} \frac{(p-1) p^a}{p^{2 a }} + \frac{p^{i_0-1}}{p^{2(i_0-1)}} \sum_{l=1}^{p-1} \mathrm{e}(l(x_{i_0}-y_{i_0})p^{-1})\right)\\
& = & 1+ \left(\sum_{a=0}^{i_0-2} \frac{p-1}{p^{a}} - \frac{1}{p^{i_0-1}} \right)\\
& = & 1+ \left(p-\frac{p+1}{p^{i_0-1}} \right).
\end{eqnarray*}

Using the definition
\begin{eqnarray}\label{phi_fun}
\theta_{p}(x,y) 
&:=& \left\{ 
\begin{array}{ll} 
p & \mbox{if } x = y, \\
p - p^{1-i_0} (p + 1) & \mbox{if } x_{i_0} \neq y_{i_0} \mbox{ and } \\ & x_{i} = y_{i} \mbox{ for } i = 1,\ldots, i_0-1, 
\end{array} \right.
\end{eqnarray}
we have $$K_{p}(x,y)=1+\theta_{p}(x,y).$$ The function $\theta_{p}$ can easily be computed and therefore also the reproducing kernels $K_{p}$ and $K_{\bsp,s}$, respectively, can easily be computed. Together with \eqref{err_dia1} we obtain the following computable formula for the $\bsp$-adic diaphony.

\begin{proposition}
Let $s \in \NN$ and $\bsp \in \PP^s$. The $\bsp$-adic diaphony of a sequence $\omega =(\bsx_n)_{n \in \NN_0}$ in $[0,1)^s$, where $\bsx_n=(x_{n,1},\ldots,x_{n,s})$ for $n \in \NN_0$, can be written as
\begin{equation*}
F_N(\omega)=\left(\frac{1}{\sigma_{\bsp}-1} \left(-1+\frac{1}{N^2} \sum_{n,m=0}^{N-1} \prod_{i=1}^s(1+\theta_{p_i}(x_{n,i},x_{m,i}))\right)\right)^{1/2},
\end{equation*}
where $\theta_p$ is defined by \eqref{phi_fun}.
\end{proposition}

\begin{remark}\rm
Note that $K_p(x,y)=K_{\rm wal}(x,y)$ where $K_{\rm wal}$ is a reproducing kernel for the so-called Wash space. We refer to \cite{DP05a} for an exact definition and for some background. Hence it follows from \cite{DP05a} and from our considerations here that for $\bsp=(p,\ldots,p)$ the $\bsp$-adic diaphony and the diaphony based on Walsh functions in base $b=p$ coincide. 
\end{remark}

\section{The $\bsp$-adic diaphony of the Halton sequence}\label{dia_halt}

In this section we present the main result of this paper. We analyze the $\bsp$-adic diaphony of the Halton sequence. The {\it $s$-dimensional Halton sequence} in pairwise different prime bases $p_1,\ldots,p_s$ is defined by $\bsx_n=(\phi_{p_1}(n),\ldots, \phi_{p_s}(n))$ for $n \in \NN_0$.

\begin{theorem}\label{th1}
Let $\omega$ be the Halton sequence in pairwise different prime bases $p_1,\ldots,p_s$. Then for any $N \in \NN$ we have $$F_N^2(\omega) \le c(\bsp,s) \frac{(\log N)^s}{N^2}+ \frac{d(\bsp,s)}{N^2},$$ where $c(\bsp,s)=\frac{1}{\sigma_{\bsp}-1} \frac{\pi^2}{3}\left(\prod_{j=1}^s \left(1+2 \frac{p_j^2}{\log p_j}\right)\right)$ and $d(\bsp,s)=2 s \max_{1 \le i \le s} p_i$.
\end{theorem}

For the proof of Theorem~\ref{th1} we need the following lemma: 

\begin{lemma}\label{bd_gammasum}
Let $\omega=(\bsx_n)_{n \in \NN_0}$ be the Halton sequence in pairwise different prime bases $p_1,\ldots,p_s$. Then for any $N \in \NN$ and any $\bsk \in \NN_0^s\setminus\{\bszero\}$ we have $$\left|\sum_{n=0}^{N-1}\gamma_{\bsk}(\bsx_n)\right| \le \frac{1}{\| \sum_{j=1}^s\phi_{p_j}(k_j)\|},$$ where $\|x\|$ denotes the distance to the nearest integer of a real $x$, i.e., $\|x\|:=\min(x-\lfloor x \rfloor , 1-(x-\lfloor x \rfloor))$.  
\end{lemma}

\begin{proof}
Again we use the notation $\mathrm{e}(x) := \exp(2\pi \mathrm{i} x)$. Since $p_1,\ldots,p_s$ are pairwise different prime numbers it follows that $\sum_{j=1}^s\phi_{p_j}(k_j) \not\in \ZZ$. Hence we have 
\begin{eqnarray*}
\sum_{n=0}^{N-1}\gamma_{\bsk}(\bsx_n)= \sum_{n=0}^{N-1} \mathrm{e}\left(n \sum_{j=1}^s \phi_{p_j}(k_j)\right)= \frac{\mathrm{e}(N \sum_{j=1}^s \phi_{p_j}(k_j))-1}{\mathrm{e}(\sum_{j=1}^s \phi_{p_j}(k_j))-1},
\end{eqnarray*}
and further
\begin{eqnarray*}
\left|\sum_{n=0}^{N-1}\gamma_{\bsk}(\bsx_n)\right| \le \frac{2}{|\mathrm{e}(\sum_{j=1}^s \phi_{p_j}(k_j))-1|}=\frac{1}{|\sin(\pi \sum_{j=1}^s \phi_{p_j}(k_j))|}\le \frac{1}{\| \sum_{j=1}^s\phi_{p_j}(k_j)\|}.
\end{eqnarray*}
\end{proof}

For the proof of Theorem~\ref{th1} we need some further notation: for $\bsg=(g_1,\ldots,g_s)\in \NN^s$ let $$\Delta_{\bsp}(\bsg)=\{\bsk=(k_1,\ldots,k_s)\in \NN_0^s\, : \, 0 \le k_i < p_i^{g_i}\mbox{ for all } 1 \le i \le s\}$$ and let $$\overline{\Delta}_{\bsp}(\bsg)=\{\bsk=(k_1,\ldots,k_s)\in \NN^s\, : \, 1 \le k_i < p_i^{g_i}\mbox{ for all } 1 \le i \le s\}.$$ 
Furthermore, let $\Delta_{\bsp}^{\ast}(\bsg)=\Delta_{\bsp}(\bsg)\setminus\{\bszero\}$.

Now we give the proof of Theorem~\ref{th1}.

\begin{proof}
It is shown in \cite[Proof of Theorem~3.24]{hel2010} that for arbitrary $\bsg=(g_1,\ldots,g_s)\in \NN^s$ we have $$F_N^2(\omega) \le \frac{1}{\sigma_{\bsp}-1} \sum_{\bsk \in \Delta_{\bsp}^{\ast}(\bsg)} \rho_{\bsp}(\bsk) \left|\frac{1}{N}\sum_{n=0}^{N-1}\gamma_{\bsk}(\bsx_n)\right|^2+ \frac{\sigma_{\bsp}-\sigma_{\bsp}(\bsg)}{\sigma_{\bsp}-1},$$ where $$\sigma_{\bsp}(\bsg)=\prod_{i=1}^s(p_i+1-p_i^{1-g_i}).$$

Let $$\Sigma:=\sum_{\bsk \in \Delta_{\bsp}^{\ast}(\bsg)} \rho_{\bsp}(\bsk)\left|\frac{1}{N}\sum_{n=0}^{N-1}\gamma_{\bsk}(\bsx_n)\right|^2.$$ 
From Lemma~\ref{bd_gammasum} we obtain
\begin{eqnarray}
\Sigma & \le & \frac{1}{N^2}\sum_{\bsk \in \Delta_{\bsp}^{\ast}(\bsg)}  \frac{\rho_{\bsp}(\bsk)}{\| \sum_{j=1}^s\phi_{p_j}(k_j)\|^2}\nonumber \\
& = & \frac{1}{N^2} \sum_{\emptyset \not= \uu \subseteq [s]} \sum_{\bsk_{\uu} \in \overline{\Delta}_{\bsp_{\uu}}(\bsg_{\uu})}  \frac{\prod_{j \in \uu}\rho_{p_j}(k_j)}{\| \sum_{j \in \uu}\phi_{p_j}(k_j)\|^2},\label{bde2_dia}
\end{eqnarray}
where $[s]:=\{1,\ldots,s\}$ and for $\uu \subseteq [s]$ and $\bsk=(k_1,\ldots,k_s)$ we write $\bsk_{\uu}=(k_j)_{j \in \uu}$ and analogously for $\bsp_{\uu}$ and $\bsg_{\uu}$.

We show that 

\begin{equation}\label{bd1k_u_dia}
\sum_{\bsk_{\uu} \in \overline{\Delta}_{\bsp_{\uu}}(\bsg_{\uu})}  \frac{\prod_{j \in \uu}\rho_{p_j}(k_j)}{\| \sum_{j \in \uu}\phi_{p_j}(k_j)\|^2} \le \frac{\pi^2}{3} \prod_{j \in \uu} g_j p_j^2.
\end{equation}
W.l.o.g. we may assume that $\uu =\{1,\ldots,t\}=:[t]$. Then we have 
\begin{eqnarray}\label{bd2k_u_dia}
\sum_{\bsk_{[t]} \in \overline{\Delta}_{\bsp_{[t]}}(\bsg_{[t]})} \frac{\prod_{j=1}^t \rho_{p_j}(k_j)}{\| \sum_{j =1}^t \phi_{p_j}(k_j)\|^2} = \sum_{u_1=0}^{g_1 -1}\ldots \sum_{u_t=0}^{g_t -1} \frac{1}{p_1^{2 u_1}\cdots p_t^{2 u_t}} \sum_{k_1=p_1^{u_1}}^{p_1^{u_1 +1}-1}\ldots \sum_{k_t=p_t^{u_t}}^{p_t^{u_t +1}-1}\frac{1}{\| \sum_{j =1}^t \phi_{p_j}(k_j)\|^2}.
\end{eqnarray}

We show that 
\begin{equation}\label{bdt_dia}
\sum_{k_1=p_1^{u_1}}^{p_1^{u_1 +1}-1}\ldots \sum_{k_t=p_t^{u_t}}^{p_t^{u_t +1}-1}\frac{1}{\| \sum_{j =1}^t \phi_{p_j}(k_j)\|^2} \le \frac{\pi^2}{3} \prod_{j =1}^t p_j^{2 u_j+2}.
\end{equation}
Any $k_j \in \{p_j^{u_j},\ldots,p_j^{u_j +1}-1\}$ has $p_j$-adic expansion of the form $k_j=\kappa_{j,0}+\kappa_{j,1} p_j+\kappa_{j,2} p_j^2+\cdots+\kappa_{j,u_j} p_j^{u_j}$ with $\kappa_{j,u_j}\not=0$ and hence we have $$\phi_{p_j}(k_j)=\frac{\kappa_{j,0}}{p_j}+\cdots+\frac{\kappa_{j,u_j}}{p^{u_j +1}} =: \frac{A_j}{p_j^{u_j+1}}.$$ Note that $A_j \in \{1,\ldots,p_j^{u_j+1}-1\}$. Hence we have 
\begin{eqnarray*}
\sum_{k_1=p_1^{u_1}}^{p_1^{u_1 +1}-1}\ldots \sum_{k_t=p_t^{u_t}}^{p_t^{u_t +1}-1}\frac{1}{\| \sum_{j =1}^t \phi_{p_j}(k_j)\|^2} & \le & \sum_{A_1=1}^{p_1^{u_1+1}-1}\ldots \sum_{A_t=1}^{p_t^{u_t+1}-1} \left\| \sum_{j=1}^t \frac{A_j}{p_j^{u_j+1}} \right\|^{-2}\\
& = & \sum_{A_1=1}^{p_1^{u_1+1}-1}\ldots \sum_{A_t=1}^{p_t^{u_t+1}-1} \left\| \left(\sum_{j=1}^t A_j \prod_{i=1 \atop i \not=j}^t p_i^{u_i+1}\right)  \frac{1}{\prod_{j=1}^t p_j^{u_j+1}} \right\|^{-2}.
\end{eqnarray*}
Assume for $1 \le j \le t$ there are $A_j,B_j \in \{1,\ldots,p_j^{u_j+1}-1\}$, such that $$\sum_{j=1}^t A_j \prod_{i=1 \atop i \not=j}^t p_i^{u_i+1}\equiv\sum_{j=1}^t B_j \prod_{i=1 \atop i \not=j}^t p_i^{u_i+1}\pmod{\prod_{j=1}^t p_j^{u_j+1}}.$$ Then for any index $j_0 \in \{1,\ldots,t\}$ we have $$(A_{j_0}-B_{j_0}) \prod_{i =1 \atop i \not=j_0}^t p_j^{u_j+1} \equiv \sum_{j=1\atop j \not= j_0}^t (B_j-A_j) \prod_{i=1 \atop i \not=j} p_j^{u_j+1}\equiv 0 \pmod{p_{j_0}^{u_{j_0}+1}}.$$ Since $\gcd(p_{j_0},p_i)=1$ for any $i \not=j_0$ it follows that $A_{j_0}-B_{j_0} \equiv 0 \pmod{p_{j_0}^{u_{j_0}+1}}$. But since $|A_{j_0}-B_{j_0}| < p_{j_0}^{u_{j_0}+1}$ it follows that $A_{j_0}=B_{j_0}$. Hence it follows that 
\begin{eqnarray*}
\lefteqn{\sum_{k_1=p_1^{u_1}}^{p_1^{u_1 +1}-1}\ldots \sum_{k_t=p_t^{u_t}}^{p_t^{u_t +1}-1}\frac{1}{\| \sum_{j =1}^t \phi_{p_j}(k_j)\|^2} \le \sum_{a=1}^{\prod_{j=1}^t p_j^{u_j+1}-1}\left\|\frac{a}{\prod_{j=1}^t p_j^{u_j+1}} \right\|^{-2}}\\
& = & \sum_{1 \le a < \frac{1}{2}\prod_{j=1}^t p_j^{u_j+1}} \left(\frac{a}{\prod_{j=1}^t p_j^{u_j+1}}\right)^{-2} +\sum_{\frac{1}{2}\prod_{j=1}^t p_j^{u_j+1} \le a < \prod_{j=1}^t p_j^{u_j+1}} \left(1-\frac{a}{\prod_{j=1}^t p_j^{u_j+1}}\right)^{-2}\\
& = & \prod_{j=1}^t p_j^{2 u_j+2} \left(\sum_{1 \le a < \frac{1}{2}\prod_{j=1}^t p_j^{u_j+1}}\frac{1}{a^2}+\sum_{\frac{1}{2}\prod_{j=1}^t p_j^{u_j+1} \le a < \prod_{j=1}^t p_j^{u_j+1}} \frac{1}{\left(\prod_{j=1}^t p_j^{u_j+1}-a\right)^2}\right)\\
& \le & 2 \prod_{j=1}^t p_j^{2 u_j+2} \sum_{a=1}^{\infty}\frac{1}{a^2}\\
& = & \frac{\pi^2}{3} \prod_{j=1}^t p_j^{2 u_j+2}
\end{eqnarray*}
and hence \eqref{bdt_dia} is shown.

Inserting \eqref{bdt_dia} into \eqref{bd2k_u_dia} gives
\begin{eqnarray*}
\sum_{\bsk_{[t]} \in \overline{\Delta}_{\bsp_{[t]}}(\bsg_{[t]})}\frac{\prod_{j=1}^t \rho_{p_j}(k_j)}{\| \sum_{j =1}^t \phi_{p_j}(k_j)\|^2} \le \frac{\pi^2}{3} \prod_{j=1}^t g_j p_j^2
\end{eqnarray*}
and hence \eqref{bd1k_u_dia} is shown.

Now, inserting \eqref{bd1k_u_dia} into \eqref{bde2_dia} gives 

\begin{eqnarray*}
\Sigma \le \frac{\pi^2}{3} \frac{1}{N^2} \sum_{\emptyset \not= \uu \subseteq [s]} \prod_{j \in \uu} g_j p_j^2= \frac{\pi^2}{3} \frac{1}{N^2} \left(-1+\prod_{j=1}^s \left(1+g_j p_j^2\right)\right).
\end{eqnarray*}
Now we have 
$$F_N^2(\omega) \le \frac{1}{\sigma_{\bsp}-1} \frac{\pi^2}{3} \frac{1}{N^2} \left(-1+\prod_{j=1}^s \left(1+g_j p_j^2\right)\right)+ \frac{\sigma_{\bsp}-\sigma_{\bsp}(\bsg)}{\sigma_{\bsp}-1}.$$
Choosing $g_j=\lfloor 2 \log_{p_j}N\rfloor$ we obtain 
\begin{eqnarray*}
\sigma_{\bsp}-\sigma_{\bsp}(\bsg) & = & \sigma_{\bsp}\left(1-\prod_{i=1}^s \left(1-\frac{p_i}{(p_i+1)p_i^{g_i}}\right)\right)\\
& \le & \sigma_{\bsp}\left(1-\prod_{i=1}^s \left(1-\frac{1}{p_i^{g_i}}\right)\right)\\
& \le & \sigma_{\bsp}\left(1-\prod_{i=1}^s\left(1-\frac{p_i}{N^2}\right)\right)\\
& \le & \sigma_{\bsp} \frac{s \max_{1 \le i \le s} p_i}{N^2}. 
\end{eqnarray*}
Hence
\begin{eqnarray*}
F_N^2(\omega) & \le & \frac{1}{\sigma_{\bsp}-1} \frac{\pi^2}{3} \frac{1}{N^2} \left(-1+\prod_{j=1}^s \left(1+ \frac{2 p_j^2 \log N}{\log p_j}\right)\right)+ \frac{\sigma_{\bsp}}{\sigma_{\bsp}-1}\frac{s \max_{1 \le i \le s} p_i}{N^2}\\
& \le & \frac{1}{\sigma_{\bsp}-1} \frac{\pi^2}{3}\left(\prod_{j=1}^s \left(1+ \frac{2 p_j^2}{\log p_j}\right)\right) \frac{(\log N)^s}{N^2} + \frac{2 s \max_{1 \le i \le s} p_i}{N^2}
\end{eqnarray*}
and the result follows.
\end{proof}

\paragraph{Acknowledgment:} The author would like to thank Peter Kritzer for valuable discussions.

\noindent Friedrich Pillichshammer, Institut f\"{u}r Finanzmathematik, Universit\"{a}t Linz, Altenbergstra{\ss}e 69, A-4040 Linz, Austria. Email: friedrich.pillichshammer@jku.at

\end{document}